\documentclass[a4paper]{article}
\usepackage{a4wide}
\usepackage{amssymb,amsfonts,latexsym,stmaryrd}
\usepackage[PostScript=dvips]{diagrams}
\usepackage{pstricks,pst-node,pst-tree}
\usepackage{proof}
\usepackage{QED}
\usepackage{epsfig}

\title{Logic and Categories As  Tools For Building Theories}
\author{Samson Abramsky\\
Oxford University Computing Laboratory}
\date{}

\newtheorem{theorem}{Theorem}

\newtheorem{proposition}[theorem]{Proposition}

\newcommand{\beqa}{\begin{eqnarray*}}
\newcommand{\eeqa}{\end{eqnarray*}\par\noindent}

\newcommand{\vsa}{\vspace{.1in}}

\newcommand{\YY}{\mathbf{Y}}

\newcommand{\PP}{\mathcal{P}}

\newcommand{\lrarr}{\longrightarrow}
\newcommand{\rarr}{\rightarrow}

\newcommand{\ie}{\textit{i.e.}~}
\newcommand{\CC}{\mathcal{C}}
\newcommand{\DD}{\mathcal{D}}
\newcommand{\EE}{\mathcal{E}}

\newcommand{\Rel}{\mathbf{Rel}}
\newcommand{\Set}{\mathbf{Set}}
\newcommand{\Grp}{\mathbf{Grp}}
\newcommand{\Top}{\mathbf{Top}}
\newcommand{\Mat}{\mathbf{Mat}}

\newcommand{\Term}{\mathbf{1}}

\newcommand{\id}{\mathsf{id}}
\newcommand{\Id}{\mathsf{Id}}
\newcommand{\dom}{\mathsf{dom}}
\newcommand{\cod}{\mathsf{cod}}
\newcommand{\implies}{\; \Rightarrow \;}
\newcommand{\Cat}{\mathbf{Cat}}
\newcommand{\Nat}{\mathbb{N}}
\newcommand{\ZZ}{\mathbb{Z}}
\newcommand{\RR}{\mathbb{R}}

\newcommand{\Pow}{\mathcal{P}}
\newcommand{\pow}[1]{\mathcal{P}(#1)}
\newcommand{\IFF}{\;\; \Longleftrightarrow \;\;}
\newcommand{\lsem}{\llbracket}
\newcommand{\rsem}{\rrbracket}
\newcommand{\modX}{\models_{X}}
\newcommand{\somepi}{\exists(\pi)}
\newcommand{\allpi}{\forall(\pi)}

\newarrow{Eq}=====

\begin{document}

\maketitle

\section{Introduction}

My aim in this short article is to provide an impression of some of the ideas emerging at the interface of logic and computer science, in a form which I hope will be accessible to philosophers.

Why is this even a good idea? Because there has been a huge interaction of logic and computer science over the past half-century which has not only played an important r\^ole in shaping Computer Science, but has also greatly broadened the scope and enriched the content of logic itself.

This huge effect of Computer Science on Logic over the past five decades has several aspects:
new ways of \emph{using} logic, new attitudes to logic, new questions and methods.
These lead to new perspectives on the question:
\begin{center}
What logic is --- and should be!
\end{center}

\noindent Our main concern is with method and attitude rather than matter; nevertheless, we shall base the general points we wish to make on a case study: \emph{Category theory}.
Many other examples could have been used to illustrate our theme, but this will serve to illustrate some of the points we wish to make.

\section{Category Theory}

Category theory is a vast subject. It has enormous potential for any serious version of `formal philosophy' --- and yet this has hardly been realized.

We shall begin with introduction to some basic elements of category theory, focussing on the fascinating conceptual issues which arise even at the most elementary level of the subject, and then discuss some its consequences and philosophical ramifications.

\subsection{Some Basic Notions of Category Theory}
We briefly recall the basic definitions. A category has a collections of \emph{objects} $A$, $B$, $C$, \ldots, and a collection of \emph{arrows} (or \emph{morphisms}) $f$, $g$ $h$ \ldots . Each arrow has specified objects as its \emph{domain} and \emph{codomain}.\footnote{More formally, there are operations $\dom$, $\cod$ from arrows to objects.} We write $f : A \rarr B$ for an arrow with domain $A$ and codomain $B$. For any triple of objects $A, B, C$ there is an operation of \emph{composition}: given $f : A \rarr B$ and $g : B \rarr C$, we can form $g \circ f : A \rarr C$. Note that the codomain of $f$ has to match with the domain of $g$. Moreover, for each object $A$, there is an \emph{identity arrow} $\id_A : A \rarr A$. These data are subject to the following axioms:
\[ h \circ (g \circ f) = (h \circ g) \circ f  \qquad f \circ \id_{A} = f = \id_{B} \circ f \]
whenever the indicated compositions make sense, \ie the domains and codomains match appropriately.

These definitions appear at first sight fairly innocuous: some kind of algebraic structure, reminiscent of monoids (groups without inverses), but with the clumsy-looking apparatus of objects, domains and codomains restricting the possibilities for composition of arrows. These first appearances are deceptive, as we shall see, although in a few pages we can only convey a glimpse of the richness of the notions which arise as the theory unfolds.

Let us now see some first examples of  categories.
\begin{itemize}
\item The most basic example of a category is $\Set$: the objects are sets, and the arrows are functions. Composition and identities have their usual meaning for functions.
\item {Any kind of mathematical structure, together with structure preserving functions, forms a category. E.g.}
\begin{itemize}
\item {\textbf{Mon} (monoids and monoid homomorphisms)}
\item {\textbf{Grp} (groups and group homomorphisms)}
\item {$\mbox{\textbf{Vect}}_{k}$ (vector spaces over a field $k$, and linear maps)}
\item {\textbf{Pos} (partially ordered sets and monotone functions)}
\item {\textbf{Top} (topological spaces and continuous functions)}
\end{itemize}
\item $\Rel$: objects are sets, arrows $R : X \rightarrow Y$ are \emph{relations} $R \subseteq X \times Y$.
Relational composition:
\[ R;S (x, z) \;\; \Longleftrightarrow \;\; \exists y. \, R(x, y) \; \wedge \; S(y, z) \]
\item Let $k$ be a field (for example, the real or complex numbers). Consider the following category $\mbox{\textbf{Mat}}_{k}$. The objects are natural numbers. A morphism $M : \mathbf{n} \rarr \mathbf{m}$ is an $\mathbf{n} \times \mathbf{m}$ matrix with entries in $k$. Composition is matrix multiplication, and the identity on $\mathbf{n}$ is the $\mathbf{n} \times \mathbf{n}$ diagonal matrix.
\item Monoids are one-object categories. Arrows correspond to the elements of the monoid, composition of arrows to the monoid multiplication, and the identity arrow to the monoid unit.
\item {A category in which for each pair of objects $A$, $B$ there is at most one morphism from $A$ to $B$ is the same thing as a
  \emph{preorder}, \ie a reflexive and transitive relation.} Note that the identity arrows correspond to reflexivity, and composition to transitivity.
\end{itemize}

\subsubsection{Categories as Contexts and as Structures}
Note that our first class of examples illustrate the idea of categories as \emph{mathematical contexts}; settings in which various mathematical theories can be developed. Thus for example, \textbf{Top} is the context for general topology, \textbf{Grp} is the context for group theory, etc.

This issue of ``mathematics in context'' should be emphasized. The idea that any mathematical discussion is relative to the category we happen to be working in is pervasive and fundamental. It allows us simultaneously to be both properly specific and general: specific, in that statements about mathematical structures are not really precise until we have specified which structures we are dealing with, \emph{and} which morphisms we are considering --- \ie which category we are working in. At the same time, the awareness that we are working in some category allows us to extract the proper generality for any definition or theorem, by identifying exactly which properties of the ambient category we are using.

On the other hand, the last two examples illustrate that many important mathematical structures \emph{themselves appear as categories of particular kinds}. The fact that two such different kinds of structures as monoids and posets should appear as extremal versions of categories is also rather striking.

This ability to capture mathematics both ``in the large'' and ``in the small'' is a first indication of the flexibility and power of categories.

\subsubsection{Arrows vs.~Elements}
Notice that the axioms for cateories are formulated purely in terms of the algebraic operations on arrows, without any reference to `elements' of the objects. Indeed, in general  elements are not available in a category.
We will refer to any concept which can be defined purely in terms of composition and identities  as \emph{arrow-theoretic}.
We will now take a first step towards learning to ``think with arrows'' by seeing how we can replace some familiar definitions for functions between sets couched in terms of elements by arrow-theoretic equivalents.

We say that a function $f : X \longrightarrow Y$ is:
\begin{center}
\begin{tabular}{lll}
\emph{injective} & if &
$\forall x,x'\in X.\;f(x) = f(x') \;\; \Longrightarrow \;\; x = x'$\,, \\
\emph{surjective} & if &
$\forall y \in Y. \, \exists x \in X. \, f(x) = y$\,, \\
\\
\emph{monic} & if &
$\forall g,h : Z \rarr X.\;f \circ g = f \circ h \;\; \Longrightarrow \;\; g = h$\,, \\
\emph{epic} & if & $\forall g,h : Y \rarr Z.\;g \circ f = h \circ f \;\; \Longrightarrow \;\; g = h$\,.
\end{tabular}
\end{center}
Note that injectivity and surjectivity are formulated in terms of elements, while epic and monic are arrow-theoretic.
\begin{proposition} Let $f:X\rarr Y$. Then:
\begin{enumerate}
\item $f$ is injective iff $f$ is monic.
\item $f$ is surjective iff $f$ is epic.
\end{enumerate}
\end{proposition}
\proof We show 1. Suppose $f : X \rarr Y$ is injective, and that $f \circ g = f \circ h$, where $g, h : Z \rarr X$. Then for all $z \in Z$:
\[ f(g(z)) = f \circ g(z) = f \circ h(z) = f(h(z))\,. \]
Since $f$ is injective, this implies $g(z) = h(z)$. Hence we have shown that
\[ \forall z \in Z. \; g(z) = h(z)\,, \]
and so we can conclude that $g = h$. So $f$ injective implies $f$ monic.
\\
For the converse, fix a one-element set $\Term = \{ \bullet \}$. Note that elements $x \in X$ are in 1--1 correspondence with functions $\bar{x} :
\Term \rarr X$, where $\bar{x}(\bullet) := x$. Moreover, if $f(x) = y$ then $\bar{y} = f \circ \bar{x}$\,. Writing injectivity in these terms, it
amounts to the following:
\[ \forall x, x' \in X. \; f \circ \bar{x} = f \circ \bar{x}' \;\; \Longrightarrow \;\; \bar{x} = \bar{x}' . \]
Thus we see that being injective is a \emph{special case} of being monic. \qed

The reader will enjoy --- and learn from --- proving the equivalence for functions of the conditions of being surjective and epic.

\subsubsection{Generality of Notions}
Since the concepts of monic and epic are defined in purely arrow-theoretic terms, \emph{they make sense in any category}. This possibility for making definitions in vast generality by formulating them in purely arrow-theoretic terms can be applied to virtually all the fundamental notions and constructions which pervade mathematics.

As an utterly elementary, indeed ``trivial'' example, consider the notion of isomorphism. What is an isomorphism \emph{in general}? On might try a definition at the level of generality of model theory, or Bourbaki-style structures, but this is really both unnecessarily elaborate, and still insufficiently general. Category theory has exactly the language needed to give a perfectly general answer to the question, in \emph{any mathematical context, as specified by a category}. An isomorphism in a category $\CC$ is an arrow $f : A \rarr B$ with a two-sided inverse: an arrow $g : B \rarr A$ such that
\[ g \circ f = \id_{A}, \qquad f \circ g = \id_{B} . \]
One can check that in $\Set$ this yields the notion of bijection; in $\Grp$ it yields isomorphism of groups; in $\Top$ it yields homeomorphism; in $\Mat_{k}$, it yields the usual notion of invertible matrix; and so on throughout the range of mathematical structures.
In a monoid considered as a category, an isomorphism is an invertible element. Thus a group is exactly a one-object category in which every arrow is an isomorphism! This cries out for generalization; and the notion of a category in which every arrow is an isomorphism is indeed significant --- it is the idea of a \emph{groupoid}, which plays a key r\^ole in modern geometry and topology.

We also see here a first indication of the \emph{prescriptive nature} of categorical concepts. Having defined a category, what the notion of isomorphism means inside that category is now \emph{fixed} by the general definition. We can observe and characterize what that notion is; if it isn't right for our purposes, we need to work in a different category.

\subsubsection{Replacing Coding by Intrinsic Properties}

We now consider one of the most common constructions in mathematics: the formation of ``direct products''. Once again, rather than giving a case-by-case construction of direct products in each mathematical context we encounter, we can express once and for all a general notion of product, meaningful in any category --- and such that, if a product exists, it is characterised uniquely up to unique isomorphism. Given a particular mathematical context, \ie a category, we can then verify whether on not the product exists in that category. The concrete construction appropriate to the context will enter only into the proof of \emph{existence}; all of the useful \emph{properties} of the product follow from the general definition. Moreover, the categorical notion of product has a \emph{normative} force; we can test whether a concrete construction works as intended by verifying that it satisfies the general definition.

In set theory, the cartesian product is defined in terms of the ordered pair:
\[ X \times Y := \{ (x, y) \mid x \in X \; \wedge \; y \in Y \} . \]
It turns out that ordered pairs can be \emph{defined} in set theory, e.g.~as
\[ ( x, y ) := \{ \{ x, y \}, y\} . \]
Note that in no sense is such a definition canonical. The essential \emph{properties} of ordered pairs are:
\begin{enumerate}
\item We can retrieve the first and second components $x$, $y$ of the ordered pair $(x, y )$, allowing \emph{projection functions} to be defined:
\[ \pi_{1} : (x, y) \mapsto x, \qquad \pi_{2} : (x, y) \mapsto y \, . \]
\item The information about first and second components completely determines the ordered pair:
\[ (x_{1}, x_{2} ) =  ( y_{1}, y_{2} ) \;\; \Longleftrightarrow \;\; x_{1} = y_{1} \; \wedge \; x_{2} = y_{2} . \]
\end{enumerate}

\noindent The categorical definition expresses these properties in arrow-theoretic terms, meaningful in any category.

Let $A,B$ be objects in a category $\CC$. A \emph{product} of $A$ and $B$ is an object $A\times B$ together with a pair of arrows
$A\lTo^{\pi_1} A\times B\rTo^{\pi_2}B$ such that for \emph{every} such triple $A\lTo^{f}C\rTo^{g}B$ there exists a
\emph{unique} morphism
\[ \langle f, g \rangle : C \longrightarrow A \times B \]
such that the following diagram commutes.
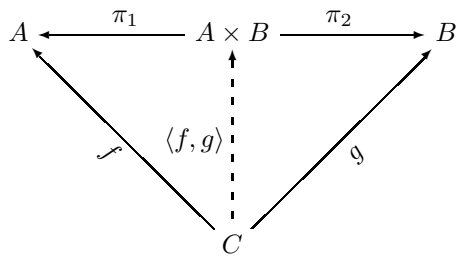
\[ \begin{diagram}[4em]
A & \lTo^{\pi_{1}} & A \times B & \rTo^{\pi_{2}} & B \\
& \luTo_{f} & \uDashto^{\langle f, g \rangle}& \ruTo_{g} & \\
& & C & &
\end{diagram}
\]
Writing the equations corresponding to this commuting diagram explicitly:
\[ \pi_{1} \circ \langle f, g \rangle = f, \qquad \pi_{2} \circ \langle f, g \rangle = g . \]
Moreover,  $\langle f, g \rangle$ is \emph{the unique morphism} $h : C \rarr A \times B$ satisfying these equations.


To relate this definition to our earlier discussion of definitons of pairing for sets, note that a `pairing' $A\lTo^{f}C\rTo^{g}B$ offers a decomposition of $C$ into components in $A$ and $B$, at the level of arrows rather than elements. The fact that pairs are uniquely determined by their components is expressed in arrow-theoretic terms by the \emph{universal property} of the product; the fact that for \emph{every} candidate pairing, there is a unique arrow  into the product, which commutes with  taking components.

As immediate evidence that this definition works in the right way, we note the following properties of the categorical product (which of course hold in \emph{any} category):
\begin{itemize}
\item The product is determined \emph{uniquely up to unique isomorphism}. That is, if there are two pairings satisfying the universal property, there is a unique isomorphism between them which commutes with taking components. This sweeps away all issues of coding and concrete representation, and shows that we have isolated the essential content of the notion of product. We shall prove this property for the related case of terminal objects in the next subsection.
\item We can also express the universal property in purely equational terms. This equational specification of products requires that we have a pairing $A\lTo^{\pi_1} A\times B\rTo^{\pi_2}B$ satisfying the equation
\[ \pi_{1} \circ \langle f, g \rangle = f, \qquad \pi_{2} \circ \langle f, g \rangle = g  \]
as before, and additionally, for any $h : C \rarr A \times B$:
\[ h = \langle \pi_1 \circ h, \pi_2 \circ h \rangle . \]
This says that any map into the product is uniquely determined by its components. This equational specification is equivalent to the definition given previously.
\end{itemize}

We look at how this definition works in some of our  example categories.
\begin{itemize}
\item In \textbf{Set}, products are the usual cartesian products.
\item In \textbf{Pos}, products are cartesian products with the pointwise order.
\item In \textbf{Top}, products are cartesian products with the product topology.
\item In $\mbox{\textbf{Vect}}_k$,  products are direct sums.
\item In a poset, seen as a category, products are \emph{greatest lower bounds}.
\end{itemize}

\subsubsection{Terminal Objects}
Our discussion in the previous sub-section was for \emph{binary} products. The same idea can be extended to define the product of any family of objects in a category. In particular, the apparently trivial idea of the product of an empty family of objects turns out to be important. The product of an empty family of objects in a category $\CC$ will be an object $\Term$; there are no projections, since there is nothing in the family to project to! The universal property turns into the following: for each object $A$ in $\CC$, there is a unique arrow from $C$ to $\Term$. Note that compatibility with the projections trivially holds, since there are no projections! This `empty product' is the notion of \emph{terminal object}, which again makes sense in any category.

\noindent \textbf{Examples}
\begin{itemize}
\item In \textbf{Set}, any one-element set $\{\bullet\}$ is terminal.
\item In \textbf{Pos}, the poset $(\{\bullet\}, \{ (\bullet,\bullet)\})$ is terminal.
\item In \textbf{Top}, the space $(\{\bullet\},\{\varnothing,\{\bullet\}\})$ is terminal.
\item In $\mbox{\textbf{Vect}}_k$, the one-element space $\{0\}$ is terminal.
\item In a poset, seen as a category, a terminal object is a greatest element.
\end{itemize}

We shall now prove that terminal objects are \emph{unique up to (unique) isomorphism}. This property is characteristic of all such ``universal'' definitions.
For example, the apparent arbitrariness in the fact that any singleton set is a terminal object in $\Set$ is answered by the fact that what counts is the property of being terminal; and this suffices to ensure that any two  objects having this property must be isomorphic to each other.

The proof of the proposition, while elementary, is a first example of distinctively categorical reasoning.
\begin{proposition}
\label{inuniqueprop}
  If $T$ and $T'$ are terminal objects in the category $\CC$ then there exists a unique isomorphism $T\cong T'$.
\end{proposition}
\begin{proof}
Since $T$ is terminal and $T'$ is an object of $\CC$, there is a unique arrow $\tau_{T'}:T' \rarr T$. We claim that $\tau_{T'}$ is  an isomorphism.
\\
Since $T'$ is terminal and $T$ is an object in $\CC$, there is an arrow $\tau'_T : T \rarr T'$. Thus we obtain $\tau_{T'} \circ \tau'_{T} : T \rarr T$, while we
also have the identity morphism $\id_{T} :T\rarr T$. But $T$ is terminal, and therefore there exists a \emph{unique} arrow from $T$ to $T$, which means that
$\tau_{T'} \circ \tau'_{T} = \id_{T}$. Similarly, $\tau'_{T} \circ \tau_{T'}=\id_{T'}$, so $\tau_{T'}$ is indeed an isomorphism. 
\end{proof}

One can reduce the corresponding property for binary products to this one, since the definition of binary product is equivalently expressed by saying that the pairing $A\lTo^{\pi_1} A\times B\rTo^{\pi_2}B$  is terminal in the category of such pairings, where the morphisms are arrows preserving the components.

It is straightforward to show that if a category has a terminal object, and all binary products, then it has products of all finite families of objects. Thus these are the two cases usually considered.

\subsubsection{Natural Numbers}
We might suppose that category theory, while suitable for formulating general notions and structures, would not work well for specific mathematical objects such as the number systems. In fact, this is not the case, and the idea of \emph{universal definition}, which we just caught a first glimpse of in the categorical notion of product, provides a powerful tool for specifying the basic discrete number systems of mathematics. We shall illustrate this with the most basic number system of all --- the natural numbers (\ie the non-negative integers).

Suppose that $\CC$ is a category with a terminal object $\Term$. We define a \emph{natural numbers object} in $\CC$ to be an object $N$ together with arrows $z : \Term \rarr N$ and $s : N \rarr N$ such that, for every such triple of an object $A$ and arrows $c : \Term \rarr A$, $f : A \rarr A$, \emph{there exists a unique arrow} (note this characteristic property of universal definitions again) $h : N \rarr A$ such that the following diagram commutes:
\[ \begin{diagram}
\Term & \rTo^{z} & N & \rTo^{s} & N \\
& \rdTo_{c} & \dTo<{h} & & \dTo>{h} \\
& & A & \rTo_{f} & A
\end{diagram}
\]
Equivalently, this means that the following equations hold:
\[ h \circ z = c, \qquad h \circ s = f \circ h . \]

Once again, the universal property implies that if a natural numbers object exists in $\CC$, it is unique up to unique isomorphism. We are not committed to any particular representation of natural numbers; we have specified the properties a structure with a constant and a unary operation must have in order to function as the natural numbers in a particular mathematical context.

In $\Set$, we can verify that $\Nat = \{ 0, 1, 2, \ldots \}$ equipped with
\[ z : \{ \bullet \} \rarr \Nat :: \bullet \mapsto 0, \qquad s : \Nat \rarr \Nat :: n \mapsto n+1 \]
does indeed form a natural numbers object. But we are not committed to any particular set-theoretic representation of $\Nat$: whether as von Neumann ordinals, Zermelo numerals \cite{Ben} or anything else. Indeed, any countable set $X$ with a particular element $x$ picked out by a map $z$, and a unary operation $s: X \rarr X$ which is injective and has $X \setminus \{ x \}$ as its image, will fulfil the definition; and any two such systems will be canonically isomorphic.

Note that, if we are given a natural numbers object $(N, z, s)$ in an abstract category $\CC$, the resources of \emph{definition by primitive  recursion} are available to us. Indeed, we can define \emph{numerals relative to $N$}: $\bar{n} : \Term \rarr N  := s^{n} \circ z$. Here $s^n$ is defined inductively: $s^1 = s$, $s^{n+1} = s \circ s^n$.\footnote{There is a metainduction going on here, using a natural number object \emph{outside} the category under discussion. This is not essential, but is a useful device for seeing what is going on.} Given any $(A, c, f)$, with the unique arrow $h : N \rarr A$ given by the universal property, we can check that $h \circ \bar{n} = f^{n} \circ c$. In fact,  if we assume that $\CC$ has finite products, and refine the definition of natural numbers object to allow for parameters, or if we keep the definition of natural numbers object as it is but assume that $\CC$ is \emph{cartesian closed} \cite{LS86}, then all primitive recursive function definitions can be interpreted in $\CC$, and will have their usual equational properties.

\subsubsection{Functors: category theory takes its own medicine}

Part of the ``categorical philosophy'' is:
\begin{center}
\emph{Don't just look at the objects; take the morphisms into account too.}
\end{center}
We can also apply this to categories!
A ``morphism of categories'' is a \emph{functor}.
A functor $F : \CC \rarr \DD$ is given by:
\begin{itemize}
\item An object map, assigning an object $FA$ of $\DD$ to every object $A$ of $\CC$.
\item An arrow map, assigning an arrow $Ff : FA \rarr FB$ of $\DD$ to every arrow $f:A\rarr B$ of $\CC$,
    in such a way that composition and identities are preserved:
\[ F(g \circ f) = Fg \circ Ff\,, \qquad F \id_A = \id_{FA]}. \]
\end{itemize}
Note that we use the same symbol to denote the object and arrow maps; in practice, this never causes confusion. 
%
The conditions expressing preservation of composition and identities are called  \emph{functoriality}.

As a first glimpse as to the importance of functoriality, note the following:
\begin{proposition}
Functors preserve isomorphisms; if $f : A \rarr B$ is an isomorphism, so is $Ff$.
\end{proposition}
\begin{proof}
Suppose that $f$ is an isomorphism, with inverse $f^{-1}$. Then
\[ F(f^{-1}) \circ F(f) = F(f^{-1} \circ f) = F(\id_A) = \id_{FA} \]
and similarly $F(f) \circ F(f^{-1}) = \id_{FB}$. So $F(f^{-1})$ is a two-sided inverse for $Ff$, which is thus an isomorphism.
\end{proof}
%
\noindent \textbf{Examples}
\begin{itemize}
\item Let $(P,\leq)$, $(Q,\leq)$ be preorders (seen as categories). A functor $F:(P,\leq)\lrarr(Q,\leq)$ is specified by an object-map, say $F:P\rarr Q$, and
an appropriate arrow-map. The arrow-map corresponds to the condition
\[ \forall p_1,p_2\in P.\, p_1\leq p_2 \implies F(p_1)\leq F(p_2)\,, \]
\ie to monotonicity of $F$. Moreover, the functoriality conditions are trivial since in the codomain $(Q,\leq)$ all hom-sets are singletons.\\
Hence, a functor between preorders is just a monotone map.

\item  Let $(M,\cdot,1)$, $(N,\cdot,1)$ be monoids. A functor $F:(M,\cdot,1)\lrarr(N,\cdot,1)$ is specified by a trivial object map (monoids are categories
with a single object) and an arrow-map, say $F:M\rarr N$. The functoriality conditions correspond to
\[ \forall m_1,m_2\in M.\, F(m_1\cdot m_2)=F(m_1)\cdot F(m_2)\,,\qquad F(1)=1\,,\]
\ie to $F$ being a monoid homomorphism.\\
Hence, a functor between monoids is just a monoid homomorphism.
\end{itemize}
Some further examples:
\begin{itemize}
\item The \emph{covariant} powerset functor $\PP:\mathbf{Set} \rarr \mathbf{Set}$:
\[ X \mapsto \PP(X)\,, \qquad (f : X \rarr Y) \mapsto \PP(f) := S \mapsto \{ f(x) \mid x \in S \} . \]
%

\item More sophisticated examples: e.g.~\emph{homology}. The basic idea of algebraic topology is that there are functorial assignments of algebraic objects (e.g.~groups) to topological spaces. The fact that functoriality implies that isomorphisms are preserved shows that these assignments are \emph{topological invariants}.
Variants of this idea (`(co)homology theories') are pervasive throughout modern pure mathematics.
\end{itemize}

\subsubsection{The category of categories}
There is a category \textbf{Cat} whose objects are categories, and whose arrows are functors. Identities in $\Cat$ are given by identity functors:
\[ \Id_{\CC}:\CC\lrarr\CC:=A\mapsto A,\, f\mapsto f. \]
Composition of functors is defined in the evident
fashion. Note that if $F:\CC \rarr \DD$ and $G : \DD \rarr \EE$ then, for $f : A \rarr B$ in $\CC$,
\[ G \circ F (f) := G(F(f)) : G(F(A))  \lrarr G(F(B)) \]
so the types work out.
A category of categories sounds (and is) circular, but in practice is harmless: one usually makes some size restriction on the categories, and then $\Cat$ will be too `big' to be an object of itself.

\subsubsection{Universality and Adjoints}

Universality arises when we are interested in finding \emph{canonical solutions} to problems of construction: that is, we are interested not just in the \emph{existence} of a solution but in its \emph{canonicity}. This canonicity should guarantee uniqueness, in the sense we have become familiar with: a canonical solution should be \emph{unique up to (unique) isomorphism}.

The notion of canonicity has a simple interpretation in the case of posets, as an \emph{extremal solution}: one that is the least or the greatest among all solutions. Such an extremal solution is obviously unique.
For example, consider the problem of finding a lower bound of a pair of elements $A$, $B$ in a poset $P$: a \emph{greatest lower bound} of $A$ and $B$ is an extremal solution to this problem. As we have seen, this is the specialisation to posets of the problem of constructing a product:
\begin{itemize}
\item A product of $A$, $B$ in a poset is an element $C$ such that $C \leq A$ and $C\leq B$, ($C$ is a lower bound);
\item and for any other solution $C'$, \ie $C'$ such that $C'\leq A$ and $C'\leq B$, we have $C'\leq C$.
    ($C$ is a greatest lower bound.)
\end{itemize}
The ideas of universality and adjunctions have an appealingly simple form in the case of posets, which is, moreover,  useful in its own right, and in particular has some striking applications to logic.
We shall develop the ideas in that special case.

Suppose $g : Q \rarr P$ is a monotone map between posets. Given $x \in P$,
a \emph{$g$-approximation of $x$} (from above) is an element $y \in Q$ such that $x \leq g(y)$.\\
A \emph{best $g$-approximation of $x$} is an element $y \in Q$ such that
\[ x \leq g(y)\ \land\ \forall z \in Q.\,(\, x \leq g(z) \implies y \leq z \,)\,. \]
If a best $g$-approximation exists then it is clearly unique.

\paragraph{Discussion} It is worth clarifying the notion of best $g$-approximation. If $y$ is a best $g$-approximation to $x$, then in particular, by monotonicity of $g$, $g(y)$ is the least element of the set of all $g(z)$ where $z \in Q$ and $x \leq g(z)$. However, the property of being a best approximation is much \emph{stronger} than the mere existence of a least element of this set. We are asking for \emph{$y$ itself} to be the least, \emph{in $Q$},  among all elements $z$ such that $x \leq g(z)$. Thus, even if $g$ is \emph{surjective}, so that for every $x$ there is a $y \in Q$ such that $g(y) = x$, there need not exist a \emph{best} $g$-approximation to $x$. This is exactly the issue of having a \emph{canonical choice} of solution.

\textbf{Exercise}
Give an example of a surjective monotone map $g : Q \rarr P$, and an element $x \in P$, such that there is no best $g$-approximation to $x$ in $Q$.

\vsa
\noindent If such a best $g$-approximation $f(x)$ exists for all $x \in P$ then we have a function $f : P \rarr Q$ such that, for all $x \in P$, $z \in Q$:
\begin{equation}\label{adjeq}
x \leq g(z) \;\; \Longleftrightarrow \;\; f(x) \leq z\, .
\end{equation}
We say that $f$ is the \emph{left adjoint} of $g$, and $g$ is the \emph{right adjoint} of $f$.
It is immediate from the definitions that the left adjoint of $g$, if it exists,  is uniquely determined by $g$.

\begin{proposition}
If such a function $f$ exists, then it is  monotone. Moreover,
\[ \id_{P} \leq g\circ f, \qquad f \circ g \leq \id_{Q}, \qquad f \circ g \circ f = f, \qquad g \circ f \circ g = g. \]
\end{proposition}
\begin{proof}  
If we take $z = f(x)$ in equation (\ref{adjeq}), then since $f(x) \leq f(x)$, $x \leq g
\circ f(x)$. Similarly, taking $x = g(z)$ we obtain  $f \circ g(z) \leq z$.
Now, the ordering on functions $h, k : P \lrarr Q$  is the \emph{pointwise order}:
\[ h \leq k \iff \forall x \in P. \, h(x) \leq k(x) . \]
This gives the first two equations.

Now, if $x \leq_{P} x'$ then $x \leq x' \leq g \circ f(x')$, so $f(x')$ is a $g$-approximation of $x$, and hence $f(x) \leq f(x')$. Thus,
$f$ is monotone.

Finally, using the fact that composition is monotone with
respect to the pointwise order on functions, and the first two equations:
\[ g = \id_{P} \circ g \leq g \circ f \circ g \leq g \circ \id_{Q} = g, \]
and hence $g = g \circ f \circ g$. The other equation is proved similarly. 
\end{proof}

\paragraph{Examples}
\begin{itemize}
\item Consider the inclusion map
\[ i : \ZZ \hookrightarrow \RR\,. \]
This has both a left adjoint $f^{L}$ and a right adjoint $f^{R}$, where $f^{L}, f^{R} : \RR \rarr \ZZ$. For all $z \in \ZZ$, $r \in \RR$:
\[ z \leq f^{R}(r)  \;\; \Longleftrightarrow \;\;  i(z) \leq r \,,\qquad f^{L}(r) \leq z  \;\; \Longleftrightarrow \;\;  r \leq i(z)\,. \]
We see from these defining properties that the right adjoint maps a real $r$ to the \emph{greatest integer below it} (the extremal solution to finding
an integer below a given real). This is the standard \emph{floor function}.\\ Similarly, the left adjoint maps a real to the least integer above it,
yielding the \emph{ceiling function}. Thus:
\[ f^{R}(r) = \lfloor r \rfloor\,, \qquad f^{L}(r) = \lceil r \rceil\,. \]

\item Consider a relation $R \subseteq  X \times Y$. $R$ induces a function:
\[ f_{R} : \pow{X} \lrarr \pow{Y}:= S \mapsto \{ y \in Y \mid \exists x \in S.\, xRy \}\,. \]
This has a right adjoint $[R] : \pow{Y} \lrarr \pow{X}$:
\[ S \subseteq [R]T \;\; \Longleftrightarrow \;\; f_{R}(S) \subseteq T\,. \]
The definition of $[R]$ which satisfies this condition is:
\[ [R]T := \{ x \in X \mid \forall y \in Y. \, xRy \; \Rightarrow \; y \in T \}\,. \]
If we consider a set of \emph{worlds} $W$ with an \emph{accessibility relation} $R \subseteq W \times W$ as in Kripke semantics for modal logic, we see
that $[R]$ gives the usual Kripke semantics for the modal operator $\Box$, seen as a propositional operator mapping the set of worlds satisfied by a
formula $\phi$ to the set of worlds satisfied by $\Box \phi$.
\\
On the other hand, if we think of the relation $R$ as the denotation of a (possibly non-deterministic) program, and $T$ as a predicate on
\emph{states}, then $[R]T$ is exactly the \emph{weakest precondition} $\mathbf{wp}(R, T)$ \cite{dijkstra1976discipline}. In \emph{Dynamic Logic} \cite{pratt2008semantical}, the two settings are combined, and
we can write expressions such as $[R]T$ directly, where $T$ will be (the denotation of) some formula, and $R$ the relation corresponding to a program.
\end{itemize}

\subsubsection{Logical notions as adjunctions}

We shall look at some examples of adjunctions arising from logic \cite{Law69}, which also give a first impression of the deep connections which exist between category theory and logic.

We begin with implication. Implication and conjunction --- whether classical or intuitionistic --- are related by the following  bidirectional inference rule:
\[ \infer={\phi \vdash \psi \rarr \theta}{\phi \; \wedge \; \psi \vdash \theta} \, . \]
If we form the preorder of formulas related by entailment as a category, this rule becomes a relationship between arrows which holds in this category. In fact, it can be shown that this \emph{uniquely characterizes} implication, and is a form of universal definition. Note that it gives the essence of what implication is. The way one proves an implication---essentially the \emph{only} way---is to add the antecedent to one's assumptions and then prove the consequent. This is justified by the above rule.

In terms of the boolean algebra of sets, define $X \Rightarrow Y = X^c \cup Y$, where $X^c$ is the set complement. Then we have, for any sets $X$, $Y$, $Z$:
\[ X \cap Y \, \subseteq \, Z \; \;\; \Longleftrightarrow \;\;\; X \, \subseteq \, Y \Rightarrow Z . \]
This says precisely that the function $f_X : Z \mapsto Z \cap X$  is left adjoint to the function $g_X : Y \mapsto X \Rightarrow Y$.

The same algebraic relation holds in any Heyting algebra, and defines intuitionistic implication.

Now we show that this same formal structure of adjoints underpins quantification.
This is the fundamental insight due to Lawvere \cite{Law69}, that \emph{quantifiers are adjoints to substitution}.

Consider a function $f : X \rarr Y$. This induces a function
\[ f^{-1} : \pow{Y} \lrarr \pow{X} :: T \mapsto \{ x \in X \mid f(x) \in T \} . \]
This function $f^{-1}$ has both a left adjoint $\exists (f) : \pow{X} \lrarr \pow{Y}$, and a right adjoint $\forall (f) : \pow{X} \lrarr \pow{Y}$. These adjoints are uniquely specified by the following conditions. For all $S \subseteq X$, $T \subseteq Y$:
\[ \exists(f)(S) \subseteq T \IFF S \subseteq f^{-1}(T), \;\; \quad  f^{-1}(T) \subseteq S \IFF  T \subseteq \forall(f)(S) . \]
The unique functions satisfying these conditions can be defined explicitly as follows:
\[ \begin{array}{lcl}
\exists(f)(S) &:= & \{ y \in Y \mid \exists x \in X. \, f(x) = y \; \wedge \; x \in S \}\,,  \\
\forall(f)(S) &:= & \{ y \in Y \mid \forall x \in X. \, f(x) = y \; \Rightarrow \; x \in S \}\,.
\end{array}
\]
\noindent Given a formula $\phi$ with free variables in $\{ v_{1}, \ldots , v_{n+1} \}$, it will receive its Tarskian denotation $\lsem \phi \rsem$ in $\Pow(A^{n+1})$ as the set of satisfying assignments:
\[ \lsem \phi \rsem = \{ s \in A^{n+1} \mid s \modX \phi \} \, . \]
We have a projection function
\[ \pi : A^{n+1} \lrarr A^{n} \; :: (a_{1}, \ldots , a_{n+1}) \mapsto (a_{1}, \ldots , a_{n}) \, . \]
Note that this projection is the Tarskian denotation of the tuple of terms $(v_{1}, \ldots , v_{n})$.
We can characterize the standard quantifiers as \emph{adjoints to this projection}:
\[ \lsem \forall v_{n+1}. \, \phi \rsem = \forall(\pi)(\lsem \phi \rsem), \qquad  \lsem \exists v_{n+1}. \, \phi \rsem = \exists(\pi)(\lsem \phi \rsem) \, . \]
More explicitly, the Tarski semantics over a structure $\mathcal{M} = (A, \ldots)$ assigns such formulas values in $\pow{A^{n+1}}$. We can regard the quantifiers $\exists v_{n+1}$, $\forall v_{n+1}$ as functions
\[ \somepi, \allpi : \pow{A^{n+1}} \lrarr \pow{A^{n}} \]
\[ \begin{array}{lcl}
\somepi(S) & = & \{ s \in A^{n} \mid \exists a \in A. \, s[v_{n+1} \mapsto a] \in S \} \\
\allpi(S) & = & \{ s \in A^{n} \mid \forall a \in A. \, s[v_{n+1} \mapsto a] \in S \}
\end{array}
\]
If we unpack the adjunction conditions for the universal quantifier, they yield the following bidirectional inference rule:
\[ \infer=[\quad X = \{ v_{1}, \ldots , v_{n} \} \, . ]{\Gamma \vdash_{X} \forall v_{n+1}. \, \phi}{\Gamma \vdash_{X} \phi} \]
Here the set $X$ keeps track of the free variables  in the assumptions $\Gamma$. Note that
the usual ``eigenvariable condition'' is automatically taken care of in this way.

Since adjoints are uniquely determined, this characterization completely captures the meaning of the quantifiers.

\subsection{Discussion: the significance of category theory}
We turn from this all too brief glimpse at the basics of category theory to discuss its conceptual significance, and why it might matter to philosophy.

The basic feature of category theory which makes it conceptually fascinating and worthy of philosophical study is that it is not just another mathematical theory, but a way of mathematical thinking, and of doing mathematics, which is genuinely distinctive, and in particular very different to the prevailing set-theoretic style which preceded it. If one wanted a clear-cut example of a paradigm-shift in the Kuhnian sense within mathematics, involving a new way of looking at the mathematical universe, then the shift from the set-theoretic to the categorical perspective provides the most dramatic example we possess.

This has been widely misunderstood. Category theory has been portrayed, sometimes by its proponents, but more often by its detractors, as offering an alternative foundational scheme for mathematics to set theory. But this is to miss the point. What category theory offers is an alternative to \emph{foundational schemes in the traditional sense themselves}. This point has been argued with great clarity and cogency in a forceful and compelling essay by Steve Awodey \cite{awodey}. We shall not attempt to replicate his arguments, but will just make some basic observations.

Firstly, it must be emphasized that the formalization of mathematics within the language of set theory, as developed in the first half of the twentieth century, has been extremely successful, and has enabled the formulation of mathematical definitions and arguments with a previously unparalleled degree of precision and rigour. However, the set-theoretical paradigm has some deficiencies.

The set-theoretical formalization of mathematics rests on the idea of representing mathematical objects as sets which can be defined within a formal set theory, typically ZFC. It is indeed a significant empirical observation, as remarked by Blass \cite{blass}, that mathematical objects can be thus represented, and mathematical proofs carried out using the axioms of set theory. This leads to claims such as the recent one by Kunen \cite{Kunen} (p. 14), that
\begin{quotation}
All abstract mathematical concepts are set-theoretic. 
All concrete mathematical objects are specific sets. 
\end{quotation}
This claim fails to ring true, for several reasons.
\begin{itemize}
\item Firstly, the set-theoretic representation is \emph{too concrete}. it involves irrelevant details and choices --- it is a coding rather than a structural representation of the concepts at hand. We saw this illustrated with the issue of defining ordered pairs in set theory, and the contrast with the categorical definition of product, which extracted the essential structural features of pairing at the right level of abstraction. Even if we think of number systems, the representation say of the natural numbers as the finite ordinals in set theory is just a particular coding --- there are many others. The essential features of the natural numbers are, rather, conveyed by the universal definition of \emph{natural numbers object} --- which makes sense in any category. We should not ask what natural numbers \emph{are}, but rather what they \emph{do} --- or what we can do with them. Set theoretic representations of mathematical objects give us too much information --- and information of the wrong kind. 
\item Furthermore, by being too specific, set theoretic representations lose much of  the generality that mathematical concepts, as used by mathematicians, naturally have. Indeed, the notion of natural numbers object makes sense in any category with a terminal object. Moreover, as a universal consruction, if it exists in a given category, it is unique up to unique isomorphism. Once we are in a particular mathematical context specified by a category, we can \emph{look and see} what the natural numbers object is --- while knowing that the standard reasoning principles such as proof by induction and definition by primitive recursion will hold.
\item When one passes to more inherently structural notions, such as `cohomology theory' or `coalgebra' the assertion that `all abstract mathematical concepts are set-theoretic' becomes staggeringly implausible, unless we replace `are' by `are codable into'. The crudity of the pure set-theoretic language becomes all too apparent. One might indeed say that insensitivity to the distortions of coding is a tell-tale feature of the set-theoretic cast of thought.

It may be useful to draw an analogy here with geometry. A major theme of 20th century geometry was the replacement of coordinate-based definitions of geometrical notions (such as tensors or varieties) with `intrinsic' definitions. Coordinates are still very useful for calculations, but the intrinsic definitions are more fundamental and more illuminating --- and ultimately more powerful. The move from set-theoretical encodings, which identify mathematical structures with specific entities in the set-theoretical universe, to universal characterizations which make sense in any mathematical context (category) satisfying some given background conditions, similarly leads to greater insight and technical power.
\end{itemize}

The foundationalist critique of category theory proceeds as follows: 
\begin{enumerate}
\item Category theory cannot emancipate itself completely from set theory, and indeed relies on set theory at certain points.
\item Hence it is not truly fundamental, and cannot serve as a foundation for mathematics. 
\end{enumerate}

On the first point, one can discern two main arguments.
\begin{itemize}
\item Firstly the very definition of category and functor presuppose the notion of a \emph{collection} of things, and of \emph{operations} on these things. So one needs an underlying theory of collections and operations as a substrate for category theory.

This is true enough; but the required `theory of collections and operations' is quite rudimentary. Certainly nothing like formal set theory is presupposed. In fact, the basic notions of categories are  \emph{essentially algebraic} in form \cite{freyd}; that is, they can be formalized as partial algebras, in which the domains of definition of the operations can themselves be defined equationally, in terms of operations which have already been specified. For example, if we consider composition as a partial binary operation $\mathsf{comp}$ on arrows, then $\mathsf{comp}(g, f)$ is defined just when $\cod(f) = \dom(g)$.

\item The second argument is that at various points, issues of \emph{size} enter into category theory. We saw an example of this in considering the category $\Cat$ of categories and functors. Is $\Cat$ an object of $\Cat$? To avoid such issues, one usually defines a version of $\Cat$ with some size restriction; for example, one only considers categories whose underlying collection of arrows form a set in Zermelo-Fraenkel set theory. Then $\Cat$ will be too large (a proper class) to be an object of itself.

There are various technical elaborations of this point. One can consider categories of arbitrary size in a stratified fashion, by assuming a sufficient supply of inaccessible cardinals (and hence of `Grothendieck universes' \cite{blass}). One can also formalize notions of size relative to an ambient category one is `working inside'; which actually describes what one is doing when formalizing category-thoretic notions in set theory.

Again, the point that in practice category theory is not completely emancipated from set theory is fair enough. What should be borne in mind, though, is how innocuous this residue of set theory in category theory actually is in practice. The strongly typed nature of category theory means that one rarely --- one is tempted to say `never' --- stumbles over these size issues; they serve more as a form of type-checking than as a substantial topic in their own right. Moreover, category theoretic arguments typically work generically in relation to size; thus in practice, one argument fits all cases, despite the stratification. 
\end{itemize}

All this is to say that, while category theory is not completely disentangled from set theory, it is quite misleading to see this as the main issue in considering the philosophical significance of categories. The temptation to do so comes from the foundationalist attitude expressed in (2) above. 

The form of categorical structuralism sketched by Awodey in \cite{awodey} stands in contrast to this set-theoretic foundationalism. It is a much better representation of mathematical practice, and it directs attention towards the kind of issues we have been discussing, and away from the well-worn  tracks of traditional thought in the philosophy of mathematics, which after more than a century have surely reached, and passed, the point of diminishing returns.

\subsubsection{Categories and Logic}

Our brief introduction to category theory did not reach the rich and deep connections which exist between category theory and logic. Categorical logic is a well-developed area, with several different branches. The most prominent of these is \emph{topos theory}.

Topos theory is an enormous field in its own right, now magisterially presented in Peter Johnstone's \textit{magnum opus} \cite{elephant}. Because, among other things, it provides a categorical formulation of a form of set theory, it is often seen as the main or even the only part of category theory relevant to philosophy.
Topos theory is seen as an alternative or rival to standard versions of set theory, and the relevance of category theory to the foundations of mathematics is judged in these terms.

There are many things within topos theory of great conceptual interest; but topos theory is far from covering all of categorical logic, let alone all of category theory. From our perspective, there is a great deal of `logic' in the elementary parts of category theory which we have diacussed. The overemphasis on topos theory in this context arises from the wish to understand the novel perspectives of category theory in terms of the traditional concepts of logic and set theory. This impulse is understandable, but misguided. As we have already argued, learning to look at mathematics from a category-theoretic viewpoint is a real and deep-seated paradigm shift. It is only by embracing it that we will reap the full benefits.

Thus while we heartily recommend learning about topos theory, this should build on having already absorbed the lessons to be learnt from category theory in general, and with the awareness that there are other important connections between category theory and logic, in particular \emph{categorical proof theory} and \emph{type theory}.

\subsubsection{Applications of Category Theory}

As we have argued, category theory has a great deal of intrinsic conceptual interest. Beyond this, it offers great potential for applications in formal philosophy, as a powerful and versatile tool for building theories. The best evidence for this comes from Theoretical Computer Science, which has seen an extensive development of applications of category theory over the past four decades.

Some of the main areas where category theory has been applied in Computer Science  include:
\begin{itemize}
\item \textbf{Semantics of Computation}.
Denotational semantics of programming languages makes extensive use of categories. In particular, categories of domains have been widely studied \cite{Scott70,paper30}. An important topic has been the study of \emph{recursive domain equations} such as
\[ D \cong [D \rarr D] \]
which is a space isomorphic to its own function space. Such spaces do not arise in ordinary mathematics, but are just what is needed to provide models for the type-free $\lambda$-calculus \cite{Bar}, in which one has self-application, leading to expressions such as the $\YY$ combinator
\[ \lambda f. (\lambda x.  f(x x)) (\lambda x.  f(x x)) \]
which produces fixpoints from arbitrary terms: $\YY M = M(\YY M)$.

The solution of such domain equations is expressed in terms of fixpoints of functors:
\[ FX \cong X. \]
This approach to the consistent interpretation of a large class of recursive data types has proved very powerful and expressive, in allowing a wide range of reflexive and recursive behaviours to be modelled.

Another form of categorical structure which has proved very useful in articulating the semantic structure of programs are monads. Various `notions of computation' can be encapsulated as monads \cite{DBLP:journals/iandc/Moggi91}. This has proved a fruitful idea, not only in semantics, but also in the development of functional programming languages.

\item \textbf{Type Theories}.
An important point of contact between category theory and logic is in the realm of proof theory and type theory. Logical systems can be represented as categories in which formulas are objects, proofs are arrows, and equality of arrows reflects equality of proofs \cite{LS86}. There are deep connections between cut-elimination in proof systems, and coherence theorems in category theory. Moreover, this paradigm extends to type theories of various kinds, which have played an important r\^ole in computer science as core calculi for programming languages, and as the basis for automated proof systems.

\item \textbf{Coalgebra}.
Over the past couple of decades, a very lively research area has developed in the field of \emph{coalgebra}. In particular, `universal coalgebra' has been quite extensively developed as a very attractive theory of systems \cite{DBLP:journals/tcs/Rutten00}. This entire area is a good witness to the possibilities afforded by categorical thinking. The idea of an algebra as a set equipped with some operations is familiar, and readily generalizes to the usual setting for universal algebra. Category theory allows us to \emph{dualize} the usual discussion of algebras to obtain a very general notion of \emph{coalgebras of an endofunctor}. Coalgebras open up a new and quite unexpected territory, and provides an effective abstraction and mathematical theory for a central class of computational phenomena:
\begin{itemize}
\item Programming over \emph{infinite data structures}, such as streams, lazy lists, infinite trees, etc.

\item A novel notion of \emph{coinduction}

\item Modelling \emph{state-based computations} of all kinds

\item A general notion of \emph{observation equivalence} between processes.

\item A general form of \emph{coalgebraic logic}, which can be seen as a wide-ranging generalization of modal logic.
\end{itemize}
In fact, coalgebra provides the basis for a very expressive and flexible theory of discrete, state-based dynamical systems, which seem ripe for much wider application than has been considered thus far; for a recent application to the representation of physical systems, see \cite{chucoalg}.

\item \textbf{Monoidal Categories}.

Monoidal categories impart a geometrical flavour to category theory. They have a beautiful description in terms of `string diagrams'  \cite{selinger10}, which allows equational proofs to be carried out in a visually compelling way. There are precise correspondences between free monoidal categories of various kinds, and constructions of braids, tangles, links, and other basic structures in knot theory and low-dimensional topology. Monoidal categories are also the appropriate general setting for the discussion of multilinear algebra, and, as has recently been shown, for much of the basic apparatus of quantum mechanics and quantum information: tensor products, traces, kets, bras and scalars, map-state duality, Bell states, teleportation and more \cite{AC04,AC09}. There are  also deep links to linear logic and other substructural, `resource-sensitive'  logics, and to diagrammatic representations of proofs. For a paper showing links between all these topics, see \cite{TL}.
Monoidal categories are used in the  modelling of concurrent processes \cite{bigraphs}, and are beginning to be employed in `computational systems biology' \cite{stochbigraphs}.

Altogether, the development of structures based on monoidal categories, and their use in modelling a wide range of computational, physical, and even biological phenomena, is one of the liveliest areas in current logically and semantically oriented Theoretical Computer Science.

It is interesting to compare and contrast the two rich realms of monoidal categories and the structures built upon them, on the one hand; and topos theory, on the other. One might say: the \emph{linear} world, and the \emph{cartesian} world. It is still not clear how these two worlds should be related. A clearer understanding of the mathematical and structural issues here may shed light on difficult questions such as the relation of quantum and classical in physics.
\end{itemize}

Having surveyed some of the ways in which category theory has been used within Computer Science, we shall now consider some of the features and qualities of category theory which have made it particularly suitable for these applications, and which may suggest a wider range of possible applications within the scope of formal philosophy.

\paragraph{Modelling at the right level of abstraction}
As we have discussed, category theory goes beyond coding to extract the essential features of concepts in terms of universal characterizations, which are then uniquely specified up to isomorphism.
This is not just aesthetically pleasing; as experience in Computer Science has shown, working at the right level of abstraction is \emph{essential} if large and complex systems are to be described and reasoned about in a managable fashion. Formal philosophy will benefit enormously by learning this lesson --- among others! --- from Computer Science.

\paragraph{Compositionality}
Another deep lesson to be learned from Computer Science is the importance of \emph{compositionality}, in the general sense of a form of description of complex systems in terms of their parts. This notion originates in logic, but has been greatly widened in scope and applicability in its use in computer science. 

The traditional approach to systems modelling in the sciences has been \emph{monolithic}; one considers a whole system, models it with a system of differential equations or some other formalism, and then analyzes the model.

In the compositional approach, one  starts with a fixed set of basic, simple building blocks, and \emph{constructions} for building new, more complex systems out of given sub-systems, and builds up the required complex system with these. This typically leads to some form of algebraic description of complex systems:
\[ S \; = \; \omega (S_1 , \ldots , S_n ) \]
where $\omega$ is an operation corresponding to one of the system-building constructions.

In order to understand the logical properties of such a system, one can develop a matching compositional view:
\[ \frac{S_1 \models \phi_1 , \ldots , S_n \models \phi_n}{%
 \omega (S_1 , \ldots , S_n ) \models \phi} \]
 One searches for a rule will allows one to reduce the verification of a property of a complex system to verifications of sub-properties for its components.
 
The compositional methods for description and analysis of systems which have been developed in Computer Science are ripe for application in a much wider range of scientific contexts --- and in formal philosophy.

\paragraph{Mappings between representations}

Another familiar theme in Computer Science is the need for multiple levels of abstraction in describing and analyzing complex systems, and for mappings between them. Functorial methods provide the most general and powerful basis for such mappings. Particular cases, such as Galois connections, which specialize the categorical notion of  adjoint functors to posets, are widely used in abstract interpretation \cite{DBLP:conf/popl/CousotC77}.

\paragraph{Normative criteria for definitions}
As we have already remarked on a couple of occasions, category theory has a strong normative force. If we devise a mapping from one kind of structure to another, category theory tells us that we should demand that it maps morphisms as well as objects, and that it should be \emph{functorial}. Similarly, if we devise some kind of product for a certain type of structures, category theory tells us which properties our construction should  satisfy to indeed be a product in the corresponding category. More generally, constructions, if they are `canonical', should satisfy a suitable universal property; and if they do, then they are unique up to isomorphism. There are other important criteria too, such as \emph{naturality} (which we have not discussed).

These demands and criteria to be satisfied should be seen as providing valuable \emph{guidance}, as we seek to develop a suitable theory to capture some phenomenon. If we have no such guidance, it is all too likely that we may make various ad hoc definitions, not really knowing what we are doing. As it is, once we have specified a category, there are an enormous range of well-posed questions about its structure which we can ask. Does the category have products? Other kinds of limits and colimits? Is it cartesian closed? Is it a topos? And so on.
By the time we have answered these questions, we will already know a great deal about the structure of the category, and what we can do with it. We can also then focus on the more distinctive features of the category, which may in turn lead to a characterization of it, or perhaps to a classification of categories of that kind. 

\subsection{Logic And Category Theory As  Tools For Building Theories}

The project of scientific or formal philosophy, which seems to be gathering new energy in recent times, can surely benefit from the methods and tools offered by Category theory. Indeed, it can surely not afford to neglect them.
Logic has been used as the work-horse of formal philosophy for many years, but the limitations of logic as  traditionally conceived become apparent as soon as one takes a wider view of the intellectual landscape. In particular, Computer Science has led the way in finding new ways of applying logic --- and new forms of logic and structural mathematics which can be fruitfully applied. 

Philosophers and foundational thinkers who are willing and able to grasp these opportunities will find a rich realm of possibilities opening up before them. Perhaps this brief essay, modest in scope as it is,  will point someone along this road. If so, the author will feel handsomely rewarded.

\section{Guide to Further Reading}

The lecture notes \cite{AT}  are a natural follow-up to this article.

The short book~\cite{Pie91} is nicely written and gently paced.
A very clear,  thorough, and essentially self-contained introduction to basic category theory is given in \cite{awodey2010category}.

Another very nicely written text, focussing on the connections between categories and logic, and especially topos theory, is~\cite{Gol}, recently reissued by Dover Books. 
The book~\cite{LS2} is pitched at an elementary level, but offers insights by one of the key contributors to category theory.

The text~\cite{Mac} is a classic by one of the founders of category theory. It assumes considerable background knowledge of mathematics to fully appreciate its wide-ranging examples, but it provides invaluable coverage of the key topics.
The 3-volume handbook~\cite{Bor} provides coverage of a broad range of topics in category theory.

A classic text on categorical logic and type theory is \cite{LS86}. A more advanced text on topos theory is~\cite{MacM}; while \cite{elephant} is a comprehensive treatise, of which Volume~3 is still to appear.

\bibliographystyle{plain}

\bibliography{bibfile}

\begin{thebibliography}{10}

\bibitem{TL}
S.~Abramsky.
\newblock Temperley-{L}ieb algebra: From knot theory to logic and computation
  via quantum mechanics.
\newblock In Goong Chen, Louis Kauffman, and Sam Lomonaco, editors, {\em
  Mathematics of Quantum Computing and Technology}, pages 415--458. Taylor and
  Francis, 2007.

\bibitem{chucoalg}
S.~Abramsky.
\newblock Coalgebras, {C}hu spaces, and representations of physical systems.
\newblock In J.-P. Jouannaud, editor, {\em Proceedings of the 25th Annual IEEE
  Symposium on Logic in Computer Science: LiCS 2010}, pages 211--220. IEEE
  Computer Society Press, 2010.

\bibitem{AC04}
S.~Abramsky and B.~Coecke.
\newblock A categorical semantics of quantum protocols.
\newblock In {\em Proceedings of the 19th Annual IEEE Symposium on Logic in
  Computer Science: LiCS 2004}, pages 415--425. IEEE Computer Society, 2004.

\bibitem{AC09}
S.~Abramsky and B.~Coecke.
\newblock Categorical quantum mechanics.
\newblock In K.~Engesser, D.~Gabbay, and D.~Lehmann, editors, {\em Handbook of
  Quantum Logic and Quantum Structures: Quantum Logic}, pages 261--324.
  Elsevier, 2009.

\bibitem{paper30}
S.~Abramsky and A.~Jung.
\newblock Domain theory.
\newblock In S.~Abramsky, D.~Gabbay, and T.~S.~E. Maibaum, editors, {\em
  Handbook of Logic in Computer Science}, pages 1--168. Oxford University
  Press, 1994.

\bibitem{AT}
S.~Abramsky and N.~Tzevelekos.
\newblock Introduction to category theory and categorical logic.
\newblock In B.~Coecke, editor, {\em New Structures for Physics}, Lecture Notes
  in Physics, pages 3--94. Springer-Verlag, 2011.

\bibitem{awodey}
S.~Awodey.
\newblock An answer to {G}. {H}ellman's question "{D}oes category theory
  provide a framework for mathematical structuralism?".
\newblock {\em Philosophia Mathematica}, 12:54--64, 2004.

\bibitem{awodey2010category}
S.~Awodey.
\newblock {\em {Category theory}}.
\newblock Oxford University Press, 2010.

\bibitem{Bar}
H.~P. Barendregt.
\newblock {\em The Lambda Calculus}, volume 103 of {\em Studies in Logic and
  Foundations of Mathematics}.
\newblock North-Holland, 1984.

\bibitem{Ben}
Paul Benacerraf.
\newblock What numbers could not be.
\newblock {\em The Philosophical Review}, 74:47--73, 1965.

\bibitem{blass}
Andreas Blass.
\newblock The interaction between category theory and set theory.
\newblock In J.~Gray, editor, {\em Mathematical Applications of Category
  Theory}, volume~30 of {\em Contemporary Mathematics}, pages 5--29. AMS, 1984.

\bibitem{Bor}
F.~Borceux.
\newblock {\em Handbook of Categorical Algebra Volumes 1--3}.
\newblock Cambridge University Press, 1994.

\bibitem{DBLP:conf/popl/CousotC77}
Patrick Cousot and Radhia Cousot.
\newblock Abstract interpretation: A unified lattice model for static analysis
  of programs by construction or approximation of fixpoints.
\newblock In {\em POPL}, pages 238--252, 1977.

\bibitem{dijkstra1976discipline}
E.W. Dijkstra.
\newblock {\em {A Discipline of Programming}}.
\newblock Prentice-Hall, 1976.

\bibitem{freyd}
P.~J. Freyd.
\newblock Aspects of topoi.
\newblock {\em Bulletin of the Australian Mathematical Society}, 7(1):1--76,
  1972.

\bibitem{Gol}
R.~I. Goldblatt.
\newblock {\em Topoi, the Categorial Analysis of Logic}.
\newblock North-Holland, 1984.
\newblock Reprinted by Dover Books, 2006.

\bibitem{elephant}
P.~T. Johnstone.
\newblock {\em Sketches of an Elephant: A Topos Theory Compendium. I, II}.
\newblock Oxford University Press, 2002.

\bibitem{stochbigraphs}
Jean Krivine, Robin Milner, and Angelo Troina.
\newblock Stochastic bigraphs.
\newblock In {\em Proceedings of MFPS XXIV: Mathematical Foundations of
  Programming Semantics}, volume 218 of {\em ENTCS}, page 73Ð96, 2008.

\bibitem{Kunen}
Kenneth Kunen.
\newblock {\em Foundations of Mathematics}.
\newblock College Publications, 2009.

\bibitem{LS86}
J.~Lambek and P.~J. Scott.
\newblock {\em Introduction to higher-order categorical logic}.
\newblock Cambridge University Press, 1986.

\bibitem{Mac}
S.~Mac Lane.
\newblock {\em Categories for the Working Mathematician, Second Edition}.
\newblock Springer, 1998.

\bibitem{MacM}
S.~Mac Lane and I.~Moerdijk.
\newblock {\em Sheaves in Geometry and Logic: A First Introduction to Topos
  Theory}.
\newblock Springer, 1994.

\bibitem{Law69}
F.~W. Lawvere.
\newblock Adjointness in foundations.
\newblock {\em Dialectica}, 23:281--296, 1969.

\bibitem{LS2}
F.~W. Lawvere and S.~Schanuel.
\newblock {\em Conceptual Mathematics: A First Introduction to Categories}.
\newblock Cambridge University Press, 1997.

\bibitem{bigraphs}
Robin Milner.
\newblock {\em The Space and Motion of Communicating Agents}.
\newblock Cambridge University Press, 2009.

\bibitem{DBLP:journals/iandc/Moggi91}
Eugenio Moggi.
\newblock Notions of computation and monads.
\newblock {\em Inf. Comput.}, 93(1):55--92, 1991.

\bibitem{Pie91}
Benjamin~C. Pierce.
\newblock {\em Basic Category Theory for Computer Scientists}.
\newblock MIT Press, 1991.

\bibitem{pratt2008semantical}
V.R. Pratt.
\newblock {Semantical considerations on Floyd-Hoare logic}.
\newblock In {\em Proceedings of 17th Annual Symposium on Foundations of
  Computer Science, 1976}, pages 109--121. IEEE Press, 1976.

\bibitem{DBLP:journals/tcs/Rutten00}
Jan J. M.~M. Rutten.
\newblock Universal coalgebra: a theory of systems.
\newblock {\em Theor. Comput. Sci.}, 249(1):3--80, 2000.

\bibitem{Scott70}
D.~S. Scott.
\newblock Outline of a mathematical theory of computation.
\newblock Technical report, Oxford University Computing Laboratory, 1970.
\newblock Technical Monograph PRG-2 OUCL.

\bibitem{selinger10}
Peter Selinger.
\newblock A survey of graphical languages for monoidal categories.
\newblock In B.~Coecke, editor, {\em New Structures for Physics}, Lecture Notes
  in Physics, pages 289--355. Springer-Verlag, 2011.

\end{thebibliography}

\end{document}